\documentclass[12pt]{testarticle}

\usepackage{graphicx}

\usepackage{amssymb}
\usepackage{amsthm}
\usepackage{amsmath}

\newcommand{\cay} [1] {\mbox{Cay}(#1)}
\newcommand{\zn}{\mathbb{Z}_n}

\newtheorem{thm}{Theorem}[section]
\newtheorem{cor}[thm]{Corollary}
\newtheorem{lem}[thm]{Lemma}
\newtheorem{qstn}[thm]{Problem}
\theoremstyle{definition}

\theoremstyle{definition}
\newtheorem{exmp}{Example}[section]
\theoremstyle{remark}

\begin{document}

\begin{frontmatter}


\title{Isomorphic and Nonisomorphic, Isospectral Circulant Graphs \tnoteref{label1}}
 \tnotetext[label1]{Research was supported by a Reed College Undergraduate Research Initiative Grant.}
 \author{Julia Brown \fnref{label2}}
 \ead{julia.lazenby@alumni.reed.edu}
 \fntext[label2]{Results from an undergraduate thesis presented to Reed College in May 2008 under the advising of David Perkinson.}

\begin{abstract}
New criteria for which Cayley graphs of cyclic groups of any order can be completely determined--up to isomorphism--by the eigenvalues of their adjacency matrices is presented.  Secondly, a new construction for pairs of nonisomorphic Cayley graphs of cyclic groups with the same list of eigenvalues of their adjacency matrices will be presented.

\end{abstract}


\end{frontmatter}


\section{Introduction}
\label{sec:intro}
Let \emph{spectrum} refer to the list eigenvalues of the of the adjacency matrix of a graph.  Two graphs are \emph{isospectral} if their spectra are the same.  We say that a graph, $X$, is a \emph{circulant graph} of \emph{order n} if it is a Cayley graph of a cyclic group of order $n$, written $\cay{\zn, S}$.  The set $S \subseteq \zn \backslash \{0\}$ is called the \emph{connection set} of $X$.  If $S$ is a multiset, rather than a set, of elements of $\zn$, then we say that $\cay{\zn, S}$ is a \emph{circulant multigraph} and $S$ is the \emph{connection multiset}.  This paper studies the spectra of circulant graphs and their relationship to graph isomorphisms.  

We say that a family of graphs can be characterized by its spectra if the only isospectral graphs in that family are also isomorphic.  \protect \cite{Spec} gives a list of a dozen different families of graphs that can be characterized by their spectra.  Although there has been a great deal of research dealing with the graph isomorphism problem for circulant graphs (see \protect \cite{Iso}, \protect \cite{Iso18}, \protect \cite{Iso27}, and \protect \cite{Iso35}), there is surprisingly little known about when circulant graphs can be characterized by their spectrum.  It was previously known that when circulant graphs are of prime order, their spectra determines them completely up to isomorphism.  There are also several examples proving that not all isospectral, circulant graphs must be isomorphic. (See \protect \cite{Elspas} and \protect \cite{Gods-small} for a few.)  However, that was all that was known.   The following theorem defines a new family of circulant graphs that can be characterized by their spectra.
\begin{thm}\label{thm:main}
Let $X$ be a circulant graph (or multigraph) of order $n = p_1^{r_1} p_2^{r_2} \cdots p_s^{r_s}$ where $p_1 < p_2 < \cdots <p_s$ are primes.  Let the size of the connection set (or multiset) of $X$ be $m$.  If $p_1 \geq m$ and either $s = 1$ or $p_2 > p_1 (m-1)$, then any circulant graph isospectral to $X$ must be isomorphic to $X$.
\end{thm}
Section \ref{sec:first} is devoted to proving Theorem \ref{thm:main}.  Section \ref{sec:two} of this paper is devoted to presenting a new construction for isospectral, nonisomorphic circulant graphs.  There are several methods for constructing isospectral, non-isomorphic graphs (see \protect \cite{Gods-const} for a good overview).  However, these methods do not apply to Cayley graphs.  Before 2005, the only known construction for isospectral, non-isomorphic Cayley graphs was due to Babai who gave examples for the dihedral group of order $2p$ where $p$ is prime (\protect \cite{Bab}).  In 2005, Lubotzky et al. published a construction for isospectral, non-isomorphic Cayley graphs of the group $\mbox{PSL}_d(\mathbb{F}_q)$ for every $d \geq 5$ ($d \neq 6$) and prime power $q>2$.  The construction presented in Section \ref{sec:two} of this paper is for circulant graphs on $2^rp$ vertices for any odd prime, $p$, and integer $r>2$.
%
%
\section{A New Spectral Characterization}\label{sec:first}

It is easy to verify that the adjacency matrix of any circulant graph will be circulant, meaning that the $i^{th}$ row of the adjacency matrix is the cyclic shift of the first row by $i-1$ to the right.  Since the adjacency matrices of circulant graphs have such a rigid structure, it is no surprise that there is a simple and elegant formula for the spectra of circulant graphs.  

\begin{thm}\label{thm:spec_formula}
If $X = \cay{\zn,S}$, then $Spec(X) = \{ \lambda_x \mid x \in \zn \}$ where
\begin{equation*}
\lambda_x = \sum_{s \in S} \omega^{xs}
\end{equation*}
and $\omega$ is a fixed, primitive $n^{th}$ root of unity.
\end{thm}

\begin{proof}
Let $T$ be the linear operator corresponding to the adjacency matrix of a circulant graph $X=\cay{\zn, \{a_1, a_2, \cdots, a_m\}}$.  If  $f$ is any complex function on the vertices of $X$ we have
\begin{equation*}
T(f)(x) = f(x+a_1) + f(x+a_2) + \cdots +f(x+a_m).
\end{equation*}
Let $\omega$ be a primitive $n^{th}$ root of unity and let $g(x) = \omega^{i\,x}$ for some $i\in \zn$.  Then,
\begin{eqnarray*}
T(g)(x) &=& \omega^{ix+ia_1} + \omega^{ix+ia_2} + \cdots + \omega^{ix+ia_m}\\
&=&\omega^{ix}\left( \omega^{ia_1} + \omega^{ia_2} + \cdots + \omega^{ia_m} \right).
\end{eqnarray*}
Thus, $g$ is an eigenfunction and $\omega^{ia_1} + \omega^{ia_2} + \cdots + \omega^{ia_m}$ is an eigenvalue.
\end{proof}

\subsection{Terms and Results for a Related Group Ring}\label{sec:groupring}

Consider the group $G = \langle z \mid z^n = 1 \rangle$, and let $\omega$ be a (fixed) primitive $n^{th}$ root of unity.  Let $\varphi : \mathbb{Z}G \rightarrow  \mathbb{Z}[ \omega ]$, be defined by $\varphi (z) = \omega$.  An element of $\mathbb{Z}G$ can be uniquely written as $\alpha = \sum_{i = 0}^{n-1}C_i z^i$.  We will call this representation \emph{normal form}.  We will discuss coefficients $C_j$ for values of $j$ that may be greater than $n$.  In these cases, $C_j$ refers to $C_i$ where $i \equiv j \mod n$ and $i \in \{0,1, \ldots, n-1\}$.  Let $\varepsilon (\alpha) = \sum_{i=0}^{n-1}C_i$.  The number of nonzero coefficients is denoted by $\varepsilon_0 (\alpha)$.  Let $\mathcal{S}(\alpha)$, the \emph{support} of $\alpha$, denote the multi-set of elements of $G$ where the multiplicity of $z^i \in \mathcal{S}(\alpha)$ is $C_i$.

For any finite subset $H \subseteq G$, let $\sigma (H) = \sum_{h \in H} h$.  Two basic properties of $\sigma (H)$ are that $\varepsilon (\sigma (H)) = \varepsilon_0 ( \sigma (H)) = |H |$ (the cardinality of H), and that, if $H$ is a subgroup, $\sigma (H)  h = \, \sigma(H)$ for any $h \in H$.  If $H$ is not the trivial group and  $h$ is not an identity element, this property still holds.  So, we must have $\sigma (H) \in \ker (\varphi)$, since $\varphi(h) \neq 1$  and $\mathbb{Z}[ \omega ]$ is an integral domain.\\

\begin{lem}\label{lem:const}

If $H$ is a subgroup of $G$, then the ideal $\mathbb{Z}G \sigma (H)$ consists of all $\sum c_g g$ such that $c_g$ is constant on the cosets of $H$.\\
\end{lem}

\begin{proof}

Let $\alpha = \sum_{g \in G} b_g g$, and $\alpha \sigma (H)= \sum_{g \in G} c_g \, g$.  Then, for each $x \in G$ we have
 \begin{equation}\nonumber
 c_x = \sum_{g \in G, h \in H, g  h = x} b_g  = \sum_{h \in H} b_{xh^{-1}}.
 \end{equation}
 Letting $\pi \in H$,
 \begin{eqnarray*}
 c_{x \pi} &=& \sum_{h \in H} b_{x (\pi h^{-1})}\\
 &=& \sum_{h \in H} b_{x h}\quad .
\end{eqnarray*}
This shows that if $\sum_{g\in G}c_g g \in \mathbb{Z} G\sigma(H)$, then the $c_g$ are
constant on the cosets of $H$.  The converse is an easy exercise.
\end{proof}

Let $n = p_1^{r_1} p_2^{r_2} \cdots p_s^{r_s}$ where $p_1 < p_2 < \cdots <p_s$ are primes.
Let $P_i$ be the unique subgroup of $G$ with order $p_i$.  Theorem 3.3 of Lam and Leung's paper, \protect \cite{Lam}, reads as follows:\\
\begin{quote} \textit{
(1) If $s = 1$, $\mathbb{N}G \cap \ker (\varphi) = \mathbb{N}  \sigma (P_1)$.  (2)  If $s = 2$, $\mathbb{N}G \cap \ker (\varphi) = \mathbb{N}P_1  \sigma (P_2) + \mathbb{N}P_2  \, \sigma(P_1)$.}
\\
\end{quote}
However, the following example proves this theorem is misstated.
\begin{exmp}\label{ex:wronglam}
Let $n = 12$ and $\omega$ be a primitive $12^{th}$ root of unity.  Thus, $P_1 = \{1, z^6 \}$ and $P_2 = \{1, z^4, z^8 \}$.  The sum $z + z^5 + z^{9}$ is not an element of $\mathbb{N}P_1  \sigma (P_2) + \mathbb{N}P_2  \, \sigma(P_1)$, but it is an element of $\mathbb{N}G \cap \ker (\varphi)$ since $\omega(1 + \omega^{4} + \omega^8 ) = 0$.
\end{exmp}
The following restatement is proved using the proof supplied by Lam and Leung for Theorem 3.3  of \protect \cite{Lam}.
\begin{lem}\label{lem:corlam}
(1) If $s = 1$, $\mathbb{N}G \cap \ker (\varphi) = \mathbb{N}G   \sigma (P_1)$.  (2)  If $s = 2$, $\mathbb{N}G \cap \ker (\varphi) = \mathbb{N}G  \sigma (P_2) + \mathbb{N}G  \, \sigma(P_1)$.
\end{lem}
Thus, we can see that if $s <3$, $\mathbb{N}G \cap \ker (\varphi ) = \sum_i \mathbb{N}G  \sigma (P_i)$.  Corollary 4.9 of the same paper, \protect \cite{Lam}, gives information for when $s\geq 3$.  Corollary 4.9 reads as follows:\\
\begin{quote} \textit{
Any element $u \in \mathbb{N}G \cap \ker (\varphi )$ with $\varepsilon_0 (u) < p_1(p_2-1)+p_3-p_2$ lies in $\sum_i \mathbb{N}G  \sigma (P_i)$}.\\
\end{quote}
This corollary will have an important role in proving the following lemma.\\

\begin{lem}\label{lem:notsamesym}
Let $\alpha$ and $\beta$ be elements of $\mathbb{N} G$ such that $\varepsilon ( \alpha)  = \varepsilon( \beta ) = m$ and $ \mathcal{S}(\alpha) \cap \mathcal{S} (\beta) = \emptyset $.  If $ m \leq p_1$ and either $s=1$ or $p_2 > p_1(m-1)$, then
\begin{equation*}
\varphi (\alpha) = \varphi(\beta) \Rightarrow  \alpha = g_\alpha \sigma (P_1) \mbox{ and } \beta = g_\beta \, \sigma(P_1) 
\end{equation*}
for any $g_\alpha \in \mathcal{S}(\alpha)$ and $g_\beta \in \mathcal{S}(\beta)$.
\end{lem}

\begin{proof}
Let $\alpha =  z^{a_1}+ z^{a_2}+ \cdots +z^{a_m} $ and $ \beta =   z^{b_1}+z^{b_2}+ \cdots + z^{b_m} $.  Since  $ \mathcal{S}(\alpha) \cap \mathcal{S} (\beta) = \emptyset $, it must be the case that $z^{a_i} \neq z^{b_j}$ for any $i, j$ pair.  Therefore, for $\varphi(\alpha) = \varphi(\beta)$, $m$ must be greater than one.  Using the fact that $z^{a_i} \sigma (P_1) \in \ker (\varphi ) \cap \mathbb{N} G$ for all $i$, we can deduce the following:
\begin{eqnarray*}
0&=&\varphi ( \alpha \, \sigma(P_1)) \\
 &=& \varphi ( \alpha) + \varphi \left( \alpha \, \sigma(P_1 \backslash \{1\})\right) \\
 &=& \varphi ( \beta) + \varphi \left( \alpha \, \sigma(P_1 \backslash \{1\})\right) \\
 &=& \varphi \left( \beta +  \alpha \, \sigma(P_1 \backslash \{1\}) \right) .
\end{eqnarray*}
Let $\gamma = \beta + \alpha \, \sigma(P_1 \backslash \{1\}) \in \ker (\varphi ) \cap \mathbb{N} G$.  Now, we will show that $\gamma \in \sum_i \mathbb{N}G  \, \sigma(P_i)$.  Recall that if $p_3$ does not exist, $\gamma \in \sum_i \mathbb{N}G \, \sigma(P_i)$.  Assuming that $p_3$ does exist, we must also assume that $p_2 > p_1 (m-1)$.  Therefore, 
\begin{eqnarray*}
\varepsilon_0 (\gamma) &\leq& \varepsilon (\gamma)\\  
& = & \varepsilon (\beta) + \varepsilon \left( \alpha  \, \sigma(P_1 \backslash \{1\}) \right) \\
& = & m + m(p_1 -1)\\
& = & p_1m\\
& \leq & (p_1)^2\\
& \leq & (p_1)^2(m-1)\\
& <& (p_1)^2(m-1)+2\\
& = & p_1(p_1(m-1)) + (p_2 +2) -p_2\\
& \leq & p_1(p_2 -1)+p_3 -p_2.
\end{eqnarray*}
By Corollary 4.9 of Lam and Leung's paper, $\gamma \in \sum_i \mathbb{N}G  \, \sigma(P_i)$.  Thus, whether or not $p_3$ exists, $\gamma \in \sum_i \mathbb{N}G  \, \sigma(P_i)$, and we can write $\gamma = \sum_{i=1}^s \sum_{g \in G} x_{i,g}\,g\, \sigma (P_i)$. 
Supposing $x_{2,h} \geq 1$ for some $h \in G$, we can express $\varepsilon(\gamma)$ in two different ways: 
\begin{equation*}
x_{2, h}p_2 + n_1p_1 + n_2p_2 + \cdots + n_sp_s =\varepsilon (\gamma) =mp_1
\end{equation*}
for some $n_i \in \mathbb{N}$.  We can again use the hypotheses that $p_1 \geq m$ and $p_2> p_1(m-1)$ and thus deduce:
\begin{eqnarray*}
n_1p_1 +n_2p_2 + \cdots + n_sp_s & = &mp_1 - x_{2,h}p_2\\
&\leq& mp_1 - p_2\\
&<& p_1.
\end{eqnarray*}
Since $p_1$ is the smallest of the primes that divide $n$, $n_i = 0$ for all $1 \leq i \leq s $. This tells us that $mp_1 = x_{2,h}\, p_2$.  This would imply that $p_2$ divides $m$, but this is a contradiction because $p_2$ is greater than $m$.  Therefore, we can conclude that $x_{2,g} = 0$ for all $g \in G$.  Similarly, we can conclude that $x_{i,g} = 0$ for all $i \geq 2$, and thus, $\gamma \in \mathbb{N}G  \, \sigma(P_1)$. 

For the remainder of this proof, let $ \gamma = \sum_{i=0}^{p_1-1} x_i z^i$ be the normal form representation of $\gamma$. And let $\mathfrak{S}(i)$ represent the following four statements:
\begin{equation*}\begin{array}{rllc}
&1)& x_{a_1} \geq i \\
&2)& i<m\\
&3) & z^{a_{1+i}} = z^{a_1 + \frac{l_in}{p_1}}& \mbox{ (for some }1\leq \ell_i < p_1) \\
&4) & a_1 \neq a_{1+i}.\\
\end{array} \end{equation*}
After arbitrarily choosing $a_1$, I will show by induction that we can recursively order the $a_i$ so that $\mathfrak{S}(i)$ is true for all $i \in \{ 1, 2, \ldots, p_1-1\}$.  

Statement (2) of $\mathfrak{S}(1)$ must be true because $m>1$.  Since $z^{a_1} \sigma (P_1 \backslash \{1\}) = z^{a_1 + \frac{n}{p_1}} + z^{a_1 + \frac{2n}{p_1}} + \cdots + z^{a_1+\frac{(p_1-1)n}{p_1}}$, we can see that $x_{a_1 + \frac{n}{p_1}} \geq 1$.   We can then use Lemma \ref{lem:const} to conclude that $ x_{a_1} \geq 1$.  Therefore, statement (1) of $\mathfrak{S}$(1) is true, and $z^{a_1} \in \mathcal{S}(\gamma)$.  Since $\mathcal{S}(\alpha) \cap \mathcal{S}(\beta) = \emptyset$, we can conclude that $z^{a_1} \notin \mathcal{S}(\beta)$, and thus we know that $z^{a_1} \in \mathcal{S}( \alpha  \, \sigma(P_1 \backslash \{ 1\}))$. For this to be true, it must be the case that $z^{a_1} \in z^{a_i}(P_1 \backslash \{1\})$ for some $i$.  We know that $i \neq 1$ because $z^{a_1}$ cannot be an element of $z^{a_1}(P_1 \backslash \{1 \})$.  Without loss of generality, we can say $z^{a_1} \in z^{a_2}(P_1 \backslash \{1 \})$.  Notice that this causes statement (4) of $\mathfrak{S}(1)$ to be satisfied.  We can also conclude that $z^{a_1} = z^{a_2 + \ell n/p_1}$ for some $\ell \neq 0$.  This then allows us to rewrite $z^{a_2}$ as $z^{a_2} = z^{a_1+(p_1 - \ell)n/p_1}$.  Letting $\ell_1 = (p_1 - \ell)$, we can see that statement (3) of $\mathfrak{S}(1)$ is also true.  Hence, $\mathfrak{S}(1)$ is true.

Now I assume that $\mathfrak{S}(i)$ is true for all $i \leq j$ for some $j < p_1 -1$ in order to show that $\mathfrak{S}(j+1)$ is also true.  In order to see that statement (1) of $\mathfrak{S}(j+1)$ is true, I will rewrite $\gamma$.  For the following equations, assume that a sum from $a$ to $b$ is zero if $b < a$.
\begin{eqnarray*}
\gamma& =& \beta + \alpha \, \sigma (P_1 \backslash \{ 1 \})\\
&=& \beta+\sum_{i=0}^{j}z^{a_{1+i}} \sigma (P_1 \backslash \{1 \})+ \sum_{i=j+2}^m z^{a_i} \sigma (P_1 \backslash \{1 \})\\
&=& \beta+ \sum_{i=0}^j z^{a_1 +\frac{ \ell_in}{p_1}}  \sigma (P_1 \backslash \{1 \}) +  \sum_{i=j+2}^m z^{a_i} \sigma (P_1 \backslash \{1 \}) \, \,  \mbox{( letting } \ell_0 = 0)\\
& =& \beta + \sum_{i=0}^j \left( z^{a_1}  \, \sigma( P_1) -z^{a_1+\frac{ \ell_in}{p_1}} \right) +\sum_{i=j+2}^m z^{a_i} \sigma (P_1 \backslash \{1 \}) \\
& = & \beta + (j+1)z^{a_1} \sigma (P_1) - \sum_{i=0}^j z^{a_1 + \frac{ \ell_i n}{p_1}} + \sum_{i=j+2}^m z^{a_i} \sigma (P_1 \backslash \{1 \})\\
& = & \beta + (j+1)\sum_{k  = 0}^{p_1-1}z^{a_1 + \frac{k  n}{p_1}}  - \sum_{i=0}^j z^{a_1 + \frac{ \ell_i n}{p_1}} + \sum_{i=j+2}^m z^{a_i} \sigma (P_1 \backslash \{1 \})
\end{eqnarray*}
This makes it easier to see that for every $0 \leq k  < p_1 -1$:
\begin{equation*}
x_{a_1+\frac{ k  n}{p_1}} \geq \begin{cases}
j &\mbox{if }k  = \ell_i \mbox{ for some } 0 \leq i \leq j\\
j+1 &\mbox{otherwise.}
\end{cases}
\end{equation*}
Since $j < p_1 -1$, there must be some $k $ such that $ x_{a_1+\frac{ k  n}{p_1}} \geq j+1$.  Due to Lemma \ref{lem:const}, we can see that $x_{a_1} \geq j+1$ as well.  Hence, statement (1) of $\mathfrak{S}(j+1)$ is true.

Due to statement (3) and (4) of $\mathfrak{S}(i)$, we know that the multiplicity of $z^{a_1}$ in $\mathcal{S} \left( \sum_{i=0}^j  (z^{a_1}  \, \sigma( P_1) -z^{a_1+\ell_i n /p_1}) \right)$ is exactly $j$.  Thus, for $x_{a_1} \geq j+1$, it must be the case that $z^{a_1} \in \mathcal{S}(\beta)$ or $z^{a_1} \in \mathcal{S} \sum_{i=j+2}^m z^{a_i} \sigma (P_1 \backslash \{1 \}) )$.  Since $\mathcal{S}(\alpha) \cap \mathcal{S}(\beta) = \emptyset$, $z^{a_1}$ must be an element of the latter support.  This implies that the sum must not be zero.  Thus, $m \geq j+2$ which causes statement (2) of $\mathfrak{S}(j+1)$ to be satisfied.  Without loss of generality, we can say $z^{a_1} \in z^{a_{j +2}} (P_1 \backslash \{1\} )$.  We can then conclude that statements (3) and (4) for $\mathfrak{S}(j+1)$ are true, and therefore $\mathfrak{S}(i)$ is true for all $i < p_1 $.   We can use statement (2) and the hypothesis that $m \leq p_1$ to conclude that $m=p_1$.

A similar process can be used to prove any of the four statements for any $a_j$, not just for $a_1$.  It is most important to note that statement (4) is true for all pairs of elements in the support of $\alpha$.  With this in mind, we can conclude the following:
\begin{equation*}
i \neq j \Rightarrow a_{1+i} \neq a_{1+j} \Rightarrow a_1 + \frac{\ell_i n}{p_1} \neq a_1 + \frac{\ell_j n}{p_1} \Rightarrow \ell_i \neq \ell_j.
\end{equation*}
Now we can rewrite $\alpha$ in terms of $z^{a_1}$:
\begin{eqnarray*}
\alpha &=&  z^{a_1}+ z^{a_2}+ \cdots +z^{a_m} \\
 & =&  z^{a_1} + z^{a_1+\frac{\ell_1n}{p_1}} + \cdots + z^{a_1 + \frac{\ell_{(m-1)}n}{p_1}} \\
 & = & z^{a_1} \sigma (P_1) \quad \mbox{ (because all }\ell_i\mbox{ are unique and $m=p_1$.)}
 \end{eqnarray*}
 Since $a_1$ was chosen arbitrarily and there is no way to distinguish between $\alpha$ and $\beta$, we can say $ \alpha = z^{a_i} \, \sigma(P_1)$ and $\beta = z^{b_i}\, \sigma(P_1)$ for any $1\leq i \leq m$. 
\end{proof}
 

\begin{cor}\label{cor:sameorsym} 
Let $G = \langle z \mid z^n =1 \rangle$ where  $n = p_1^{r_1} p_2^{r_2} \cdots p_s^{r_s}$ and $p_1 < p_2 < \cdots <p_s$ are primes. Let $\alpha$ and $\beta$ be elements of $\mathbb{N} G$ such that $\varepsilon ( \alpha)  = \varepsilon( \beta ) = m$.  Suppose
\begin{eqnarray*}
&(i)& p_1 \geq m;\\
&(ii)& \mbox{either }s=1\mbox{ or }p_2 > p_1(m-1)\\
&(iii)& \varphi (\alpha) = \varphi(\beta).
\end{eqnarray*}
Then, we either have [1] $\alpha = \beta$ or [2] $m=p_1$, $\alpha = g_\alpha \sigma (P_1)$  for any $g_\alpha \in \mathcal{S}(\alpha)$, and $\beta = g_\beta \sigma (P_1)$ for any $g_\beta \in \mathcal{S}(\beta)$.
\end{cor}

\begin{proof}
In $\mathbb{N}G$, let $\alpha = \Tilde{\alpha} + \alpha'$ and $\beta = \Tilde{\beta} + \beta'$ such that $\alpha' = \beta'$ and $\mathcal{S}(\Tilde{\alpha}) \cap \mathcal{S}(\Tilde{\beta}) = \emptyset$.  If $\tilde{\alpha} = \tilde{\beta} = 0$, then $\alpha = \alpha' = \beta' = \beta$ and the proof is finished.  For the rest of this proof, assume that $\tilde{\alpha}\neq 0$.  Since $\varphi(\alpha) = \varphi(\beta)$ and $\varphi(\alpha') = \varphi(\beta')$, we have $\varphi(\Tilde{\alpha}) = \varphi (\Tilde {\beta})$.  We can then use Lemma \ref{lem:notsamesym} to  conclude that $\Tilde{\alpha} = g_\alpha\, \sigma(P_1)$ and $\Tilde{\beta} = g_\beta \, \sigma(P_1)$ for some $g_\alpha \in \mathcal{S}(\alpha), g_\beta \in \mathcal{S}(\beta)$.  Since $\tilde{\alpha} \neq 0$ we know that $\varepsilon (\Tilde{\alpha}) \neq 0$.  Thus, we have
\begin{eqnarray*}
m = \varepsilon (\alpha) & = & \varepsilon (\Tilde{\alpha}) + \varepsilon (\alpha')\\
& = &\varepsilon (z^{a_i} \sigma (P_1)) + \varepsilon (\alpha') \\
& = & p_1 + \varepsilon (\alpha').
\end{eqnarray*}
Since $m\leq p_1$, we conclude $\varepsilon(\alpha') = 0$, and hence $\alpha' = 0$.  Similarly, $ \beta' = 0$.  Therefore $\alpha = \Tilde{\alpha} = g_\alpha\, \sigma(P_1)$ and $\beta = \Tilde{\beta} =g_\beta \, \sigma(P_1)$.
\end{proof}

With these results, we now have the tools to prove Theorem \ref{thm:main}.

\subsection{Proof of Theorem \ref{thm:main}}

\begin{proof}
Suppose $Y$ is a circulant graph (or multigraph) which is isospectral to $X$.  The graph $Y$ must be of order $n$ as well.  From Theorem \ref{thm:spec_formula}, we can see that the largest eigenvalue of $X$ is $m$.  Thus, the largest eigenvalue of $Y$ must be $m$ as well.  This implies that $Y$ must have a connection set (or multiset) of size $m$.  We can write $X = \cay{ \zn, A}$ and $Y = \cay{ \zn, B}$ where $A = \{a_1, a_2, \cdots, a_m\}$ and $B= \{b_1, b_2, \cdots, b_m\}$.

Let $\omega$ be a primitive $n^{th}$ root of unity.  For the proof of this theorem, order the eigenvalues in the spectra of $X$ and $Y$ such that $\lambda_i$, the $i^{th}$ value in the spectrum of $X$, is $\lambda_i = \omega^{i\,a_1}+\omega^{i\,a_2} + \cdots + \omega^{i \, a_m}$, and $\mu_i$, the $i^{th}$ eigenvalue in the spectrum of $Y$, is $\mu_i = \omega^{i \, b_1} + \omega^{i \, b_2} + \cdots + \omega^{i\, b_m}$.

Since $X$ and $Y$ are isospectral, there is a $ 0\leq j < n $ such that $\lambda_1 = \mu_j$.  That is to say, $\omega^{a_1} + \omega^{a_2} + \cdots + \omega^{a_m}=\omega^{jb_1} + \omega^{jb_2} + \cdots + \omega^{jb_m}$. Letting $\varphi$ be the usual mapping from $\mathbb{Z} \langle z : z^n = 1 \rangle$ to $\mathbb{Z}[\omega]$ and $ \sigma(P_1) = \sum_{i=0}^{p_1-1}z^{i\frac{n}{p_1}}$, we can use Corollary \ref{cor:sameorsym} to conclude that either \emph{[1]} $z^{a_1}+ z^{a_2}+ \cdots+ z^{a_m}  =  z^{jb_1} + z^{jb_2} + \cdots + z^{jb_m} $ or \emph{[2]} $z^{a_1}+ z^{a_2}+ \cdots+ z^{a_m} = z^{a_i} \, \sigma(P_1)$ for any $a_i \in A$.  I wish to show that in either case, there is exists some $ t \in \zn$ and an ordering of $B$ such that $\omega^{a_i} = \omega^{t\,b_i}$ for all $1\leq i \leq m$.

\emph{Case 1.} \quad  $z^{a_1}+ z^{a_2}+ \cdots+ z^{a_m} =  z^{jb_1} + z^{jb_2} + \cdots + z^{jb_m} $.  This implies that $A = \{ jb_1, jb_2, \cdots , jb_m\}$.  Thus, letting $t=j$, there is an ordering of $B$ such that $a_i = tb_i$ for all $1\leq i \leq m$.

\emph{Case 2.} \quad $z^{a_1}+ z^{a_2}+ \cdots+ z^{a_m} = z^{a_i} \, \sigma(P_1)$ for any $a_i \in A$.  This implies that $\lambda_1 = \omega^{a_i} + \omega^{a_i +\frac{n}{p_1}} + \cdots + \omega^{a_i+(p_1-1)\frac{n}{p_1}}=0$, and $A = \{a_i , a_i + \frac{n}{p_1}, \cdots, a_i + (p_1 -1)\frac{n}{p_1} \}$.  Therefore,
\begin{equation*}
 \lambda_x =\omega^{xa_i} + \omega^{xa_i +x\frac{n}{p_1}} + \cdots + \omega^{xa_i+x(p_1-1)\frac{n}{p_1}} 
 = \begin{cases}
  0 & \mbox{if $p_1 \nmid x$}\\ 
p_1  \omega^{xa_i} & \mbox{if $p_1 \mid x$} \end{cases}
 \end{equation*}
for any $a_i \in A$.  Since $X$ and $Y$ are isospectral, $\mu_1 = 0$ or $\mu_1 = p_1 \omega^{a_1x}$ for some $x \in \zn$.  If $\mu_1 = p_1 \omega^{xa_1}$, then $B = \{ xa_1, xa_1, \cdots, xa_1 \}$ and $\mu_y$ will not equal zero for any $y \in \zn$.  This cannot be the case since $\mu_j = \lambda_1 = 0$.  Therefore, $\mu_1 = 0 $.  By Corollary \ref{cor:sameorsym}, we can conclude that $\mu_1 = \varphi(z^{b_i} \sigma (P_1))$, $B = \{b_i, b_i +\frac{n}{p_1}, \cdots , b_i + (p_1-1)\frac{n}{p_1} \}$, and 
\begin{equation}
\label{eq:mu}
 \mu_y = \begin{cases} 0 & \mbox{if }  p_1 \nmid y \\
 p_1  \omega^{yb_i}& \mbox{if }  p_1 \mid y
   \end{cases}
     \end{equation}
for any $b_i \in B$.

 We know that there must be some $y$ such that $\mu_y = \lambda_{p_1}=p_1 \omega^{p_1a_i}$.  By equation (\ref{eq:mu}) we know that $p_1 \mid y$.  Letting $tp_1=y$ we have:
\begin{eqnarray*}
\lambda_{p_1} &=& \mu_{tp_1}\Rightarrow \\
p_1  \omega^{p_1a_i} & = & p_1  \omega^{p_1tb_i} \Rightarrow \\
(\omega^{p_1a_i})^{1/p_1} & = & (\omega^{p_1tb_i})^{1/p_1} \Rightarrow \\
\omega^{a_i} & = & \omega^{tb_i} \zeta \quad  \quad \mbox{(where } \zeta \mbox{ is a }{p_i}^{th}\mbox{ root of unity)} \\
 & = & \omega^{tb_i + h\frac{n}{p_1}} \quad \mbox{ (for some } 0 \leq h < p_1) \\
& =& \omega^{tb_k}
\end{eqnarray*}
for any $a_i \in A$ and some $b_k \in B$.  We can reorder $B$ such that $\omega^{a_i} = \omega^{tb_i}$.

In either case, we can order $B$ such that $\omega^{a_i} = \omega^{tb_i}$ for all $1\leq i \leq m $ and some $t$.  Similarly, there is a reordering of $B$ (which may be different than the ordering just mentioned) such that for some $k \in \zn$, $\omega^{k\,a_i} = \omega^{b_i}$ for all $i$.  For the remainder of this proof, we will assume that $B$ is ordered in such a way that $\omega^{a_i} = \omega^{tb_i}$ and $\omega^{k\,a_i} = \omega^{b_{\pi(i)}}$ where $\pi$ is a permutation of $\{ 1, 2, \ldots , m \} $.  For each $1 \leq i \leq m$ there must be some $\ell \leq m$ such that $\pi^\ell(i) = i$.  Thus, we have
\begin{equation*}
\omega^{a_i k^\ell t^{\ell -1}} = \omega^{b_{\pi(i)} k^{\ell -1} t^{\ell -1}} = \omega^{a_{\pi(i)} k^{\ell -1} t^{\ell -2}} = \omega^{a_{\pi^{\ell -1}(i)}k} = \omega^{b_{\pi^\ell (i)}} = \omega^{b_i}.
\end{equation*}
Since it is also true that  $\omega^{a_i} = \omega^{t\,b_i}$, it must be the case that $(a_i , n) = (b_i, n)$ for all $1 \leq i \leq m$.

Let $g_i = (a_i,n) = (b_i,n)$, $g=(g_1, g_2, \cdots, g_m)$, and $d=(t,g)$.  Since $\omega^{a_i} = \omega^{tb_i}$ we can conclude that $(b_i,n) = (a_i, n ) = (tb_i, n)$ for all $i$.  Thus, $(t, n/g_i)=1$ for all $i$.  This implies that $(t,n/g)=1$.  Let $d = {p_1}^{r_1}{p_2}^{r_2}\cdots {p_k}^{r_k}$ be the prime factorization of $d$.  Since $d$ divides $g$, we can write the prime factorization of $g$ as $g = {p_1}^{s_1}{p_2}^{s_2} \cdots {p_k}^{s_k}{q_1}^{u_1}{q_1}^{u_2} \cdots {q_{\ell}}^{u_{\ell}}$ where $s_i \geq r_i$ for all $1 \leq i \leq k$.  Let $f = {p_1}^{s_1}{p_2}^{s_2} \cdots {p_k}^{s_k}$.  We can see that $(t,g/f)=1$.

Let $\tau = t+n/f = t + \frac{n}{g} \frac{g}{f}$.  Notice that $(\tau, n/g) = 1$, $(\tau, g/f)=1$, and $(\tau, f) =1$.  Therefore, $(\tau, n) =1$.  We can also see that
\begin{eqnarray*}
\tau b_i &=& (t+\frac n f )b_i\\
&\equiv& tb_i \mod n\\
&\equiv& a_i \mod n\\
\end{eqnarray*}
for all $1\leq i \leq m$.  Now we can define a graph isomorphism, $\psi$, by $\psi (v) = \tau v$  where $v$ is a vertex of a Cayley graph of $\zn$.  Using this isomorphism, we have
\begin{eqnarray*}
Y & \cong & \psi (Y)\\
& = & \cay{ \zn, \{ \tau b_1, \tau b_2, \cdots, \tau b_m\}} \\
& = & \cay{\zn, \{ a_1, a_2, \cdots, a_m \}}\\
& = & X.
\end{eqnarray*}
\end{proof}

\begin{cor}\label{cor:m=2}
Circulant graphs (multigraphs) with connection sets (multisets) containing only one or two elements are characterized by their spectra.
\end{cor}
Cvetkovi\'c proved a similar theorem in his doctoral thesis.  He proved that any 2-regular undirected graph is characterized by its spectrum \protect \cite{Cv2}.  (The term \emph{k-regular} means a graph for which every vertex is adjacent to exactly $k$ other vertices.)  However, the theorem does not explicitly deal with undirected graphs.
%
%
\section{A New Construction}\label{sec:two}
We have just seen that some isospectral circulant graphs must be isomorphic.  This section will provide a way to construct isospectral circulant graphs that are not isomorphic.  
\subsection{Defining the Graphs}\label{sec:def}

\begin{thm}\label{thm:def}
Let $n=2^rp$, where $p$ is an odd prime and $2\leq r $.  Let $X=\cay{\zn,A}$ and $Y=\cay{\zn,B}$ where $A$ and $B$ depend on $r$ and $p$ as follows:
\begin{eqnarray*}
A&=&\{1+i2^r \mid  0\leq i \leq \frac{p-1}2 \} \cup \{1+j2^r +\frac{n}{2} \mid  1 \leq j \leq \frac{p-1}{2} \} \\
B &=& \{1-i2^r \mid  0\leq i \leq \frac{p-1}{2} \} \cup \{1-j2^r +\frac{n}{2} \mid  1 \leq j \leq \frac{p-1}{2} \}.
\end{eqnarray*}
The graphs $X$ and $Y$ are isospectral, nonisomorphic graphs.
\end{thm}

Sections \ref{sec:isospec} through \ref{sec:noniso} are dedicated to proving this theorem.  Whenever $X$ and $Y$ are referred to in this chapter, it should be assumed that $X$ and $Y$ are the graphs defined above. 

\begin{exmp}\label{expl:n=12}
Let $n=2^2 \cdot 3$.  Then we have, $A = \{1,5\} \cup \{11\}$ and $B=\{1,9\} \cup \{3\}$.  Thus, $X=\cay{\mathbb{Z}_{12}, \{1,5,11 \} }$ and $Y=\cay{\mathbb{Z}_{12}, \{1,3,9\} }$.  These two graphs are shown below.
\begin{figure}[h]\label{fig:ex_size12}
\begin{centering}
\includegraphics[width=5in]{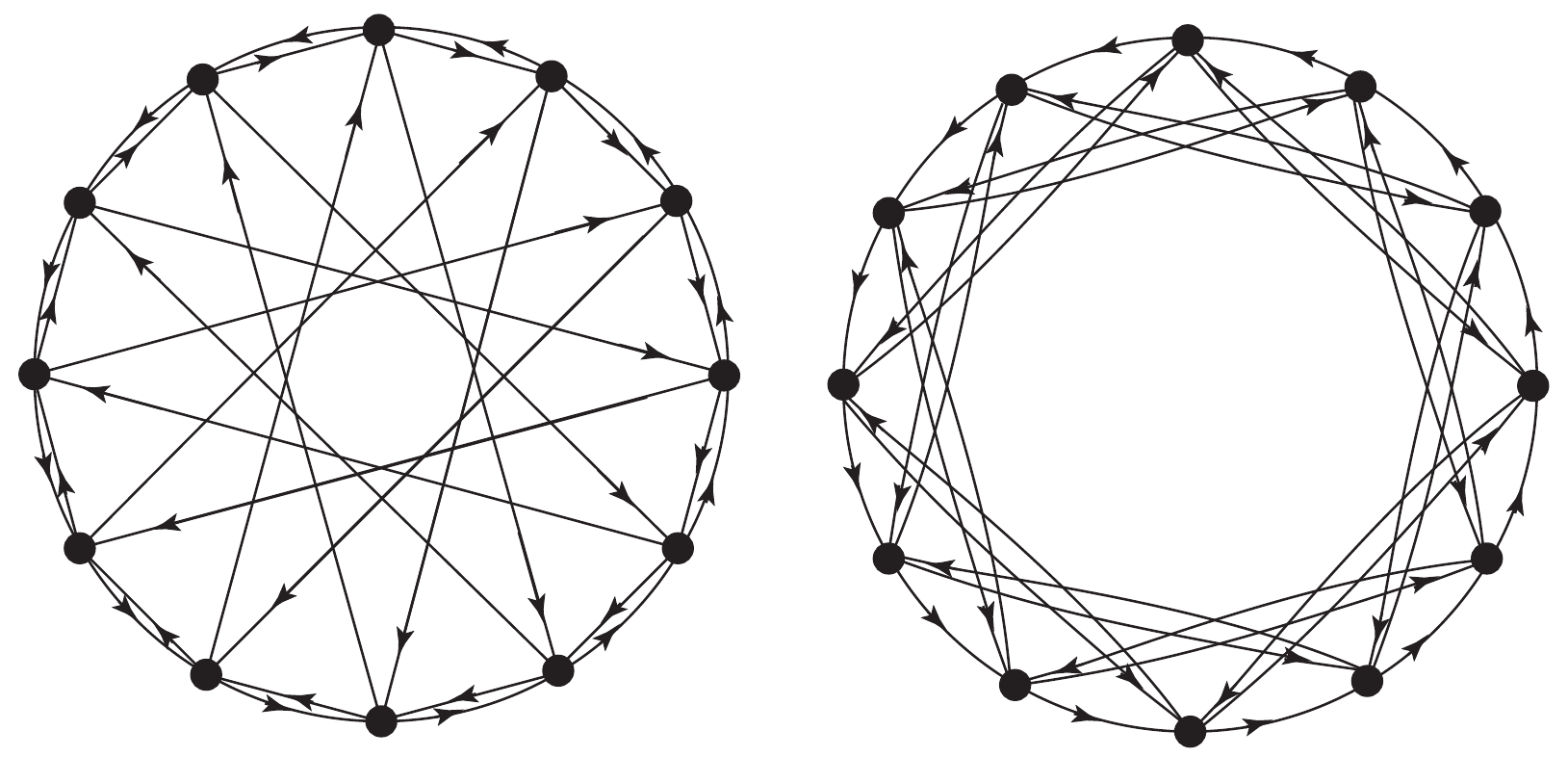}
\caption{The graphs when $n=12$.}
\end{centering}
\end{figure}
We can verify that these graphs are isospectral.  Let $\omega$ be a primitive $12^{th}$ root of unity.
\begin{equation*}
\begin{array}{ll}
\mbox{The spectrum of $X$  }           &   \mbox{The spectrum of $Y$ }\\
\omega^1+\omega^5+\omega^{11}       &    \omega^1+\omega^3+\omega^9 =\omega^1+\omega^5+\omega^{11} \\
\omega^2+\omega^{10}+\omega^{10}       &    \omega^2+\omega^6+\omega^6 = \omega^8+\omega^4+\omega^4\\
\omega^3+\omega^3+\omega^9     &    \omega^3+\omega^9+\omega^3\\
\omega^4+\omega^8+\omega^8       &    \omega^4+\omega^0+\omega^0  = \omega^{10}+\omega^2+\omega^2\\
\omega^5+\omega^1+\omega^7       &    \omega^5+\omega^3+\omega^9  = \omega^5+\omega^1+\omega^7\\
\omega^6+\omega^6+\omega^6       &    \omega^6+\omega^6+\omega^6\\
\omega^7+\omega^{11}+\omega^5       &    \omega^7+\omega^9+\omega^3  = \omega^7+\omega^{11}+\omega^5\\
\omega^8+\omega^4+\omega^4       &    \omega^8+\omega^0+\omega^0  = \omega^2+\omega^{10}+\omega^{10} \\
\omega^9+\omega^9+\omega^3       &    \omega^9+\omega^3+\omega^9\\
\omega^{10}+\omega^2+\omega^2       &    \omega^{10}+\omega^6+\omega^6  = \omega^4+\omega^8+\omega^8\\
\omega^{11}+\omega^7+\omega^1       &    \omega^{11}+\omega^9+\omega^3  = \omega^{11}+\omega^7+\omega^1\\
\omega^0+\omega^0+\omega^0       &    \omega^0+\omega^0+\omega^0
\end{array}
\end{equation*}
The spectra of these graphs are ordered in such a way to help motivate the upcoming Lemma \ref{rmrk:samespec}.
\end{exmp}

\subsection{Isospectrality}\label{sec:isospec}

The goal of this section is to show that $X$ and $Y$ are isospectral.

\begin{lem}\label{lem:sumhelp}
Let $n=2^rp$, and let $\omega$ be a primitive $n^{th}$ root of unity.  For any integer, $k \geq 0$,
\begin{equation*}
\sum_{i=0}^{(p-1)/2} \omega^{i\, 2^{r+k}} = \sum_{i=0}^{(p-1)/2} \omega^{i \, 2^r}
\end{equation*}
\end{lem}

\begin{proof}
For any $s > r$ we have
\begin{eqnarray*}
\sum_{i=0}^{p-1} \omega^{i\,2^s} &=& \sum_{i=0}^{(p-1)/2} \omega^{i\, 2^s} + \sum_{i=(p+1)/2}^{p-1} \omega^{i \, 2^s}\\
&=& \sum_{i=0}^{(p-1)/2} \omega^{i\, 2^s} + \sum_{i=0}^{(p-3)/2} \omega^{(i+\frac{p+1}{2})2^s}\\
&=& \sum_{i=0}^{(p-1)/2} \omega^{i\, 2^s} + \sum_{i=0}^{(p-3)/2} \omega^{i\, 2^s + n2^{s-r-1} +2^{s-1}}\\
&=& \sum_{i=0}^{(p-1)/2} \omega^{(2i)2^{s-1}} + \sum_{i=0}^{(p-3)/2} \omega^{(2i+1)2^{s-1}}\\
&=& \sum_{i=0}^{p-1} \omega^{i\, 2^{s-1}}
\end{eqnarray*}
By induction on the difference of $s$ and $r$, we can conclude that Lemma \ref{lem:sumhelp} is true.
\end{proof}

From now on, order the spectra of $X$ and $Y$ such that $\lambda_x$ and $\mu_x$, the $x^{th}$ eigenvalues in the spectra of $X$ and $Y$ respectively, are
\begin{eqnarray*}
\lambda_x &=&\sum_{i=0}^{(p-1)/2}\omega^{x(1+i2^r)} + \sum_{j=1}^{(p-1)/2} \omega^{x(1+j2^r+\frac{n}{2})}\\
\mu_x &=&\sum_{i=0}^{(p-1)/2}\omega^{x(1-i2^r)} + \sum_{j=1}^{(p-1)/2} \omega^{x(1-j2^r+\frac{n}{2})}\\
\end{eqnarray*}
(The spectra in Example \ref{expl:n=12} are ordered this way.)  In order to make calculations a bit clearer, I will also break down the eigenvalues of $X$ and $Y$ into two parts. Let
\begin{equation*} \begin{array}{lr}
\displaystyle \lambda_{x,\alpha} =\sum_{i=0}^{(p-1)/2}\omega^{x(1+i2^r)}, & \displaystyle \lambda_{x,\beta} = \sum_{j=1}^{(p-1)/2} \omega^{x(1+j2^r+\frac{n}{2})},\\ \\
\displaystyle \mu_{x,\alpha} = \sum_{i=0}^{(p-1)/2}\omega^{x(1-i2^r)}, \mbox{ and }  & \displaystyle \mu_{x,\beta} =  \sum_{j=1}^{(p-1)/2} \omega^{x(1-j2^r+\frac{n}{2})}.
\end{array} \end{equation*}

We can see that $\lambda_{x,\alpha} + \lambda_{x,\beta} = \lambda_x$ and $\mu_{x,\alpha} + \mu_{x,\beta} = \mu_x$.  The next lemma proves that the spectra of $X$ and $Y$ are the same.

\begin{lem}\label{rmrk:samespec}
Letting $\mu_x$ and $\lambda_x$ be as defined above, we have:
\begin{equation*}
\lambda_x = \begin{cases} \mu_{x+n/2} & \text{ if $(x,n) = 2^m$ for some $m>0$}\\
\mu_x & \text{ otherwise } \end{cases}.
\end{equation*}
\end{lem}

\begin{proof}
There are three cases based on whether $p$ or 2 divide $x$.

\emph{Case 1.} $(x,n)=1$.  \quad In this case we have 
\begin{eqnarray*}
\lambda_{x,\alpha}-\mu_{x,\beta} &=&  \sum_{i=0}^{(p-1)/2}\omega^{x(1+i2^r)} + \omega^{\frac{n}{2}} \sum_{j=1}^{(p-1)/2} \omega^{x(1-j2^r+\frac{n}{2})}\\
&=&  \sum_{i=0}^{(p-1)/2}\omega^{x+xi2^r} + \sum_{j=1}^{(p-1)/2} \omega^{x-xj2^r+n(\frac{x+1}{2})}\\
&=& \omega^x \left( \sum_{i=0}^{(p-1)/2}\omega^{(x \, i )2^r} + \sum_{j=1}^{(p-1)/2} \omega^{(-x \, j)2^r} \right)\\
&=& \omega^x \left( \sum_{i=0}^{p-1}\omega^{(x \, i )2^r} \right)\\
&=& \omega^x (0)\\
& =& 0.
\end{eqnarray*}
Thus, $\lambda_{x,\alpha} = \mu_{x,\beta}$. Similarly, $\mu_{x,\alpha} - \lambda_{x,\beta} = 0$
Therefore, $\mu_{x,\alpha} = \lambda_{x,\beta}$ and $\lambda_x = \mu_x$.\\

\emph{Case 2.} $p | x$.  \quad Letting $x = py$, we have
\begin{eqnarray*}
\lambda_x &=&\sum_{i=0}^{(p-1)/2}\omega^{py(1+i2^r)} + \sum_{j=1}^{(p-1)/2} \omega^{py(1+j2^r+\frac{n}{2})}\\
&=&\sum_{i=0}^{(p-1)/2}\omega^{py+(i y)n} + \sum_{j=1}^{(p-1)/2} \omega^{py+(jy)n+py\frac{n}{2}}\\
&=&\sum_{i=0}^{(p-1)/2}\omega^{py-(i y)n} + \sum_{j=1}^{(p-1)/2} \omega^{py-(jy)n+py\frac{n}{2}}\\
 &=&\sum_{i=0}^{(p-1)/2}\omega^{py(1-i2^r)} + \sum_{j=1}^{(p-1)/2} \omega^{py(1-j2^r+\frac{n}{2})}\\
 & = & \mu_x.
 \end{eqnarray*}

 \emph{Case 3.} $(x,n) = 2^m$ for some $m>0$. \quad Letting $x = y2^m$, where $(y,n) = 1$, we have
 \begin{eqnarray*}
 \lambda_{x,\alpha} - \mu_{x+\frac{n}{2}, \beta} &=& \sum_{i=0}^{(p-1)/2}\omega^{y2^m(1+i2^r)} + \omega^{n/2} \sum_{j=1}^{(p-1)/2} \omega^{(y2^m+\frac{n}{2})(1-j2^r+\frac{n}{2})}\\
 &=& \sum_{i=0}^{(p-1)/2}\omega^{y2^m+i\,y2^{r+m}} + \sum_{j=1}^{(p-1)/2} \omega^{y2^m - j\,y2^{r+m}}\\
 &=& \sum_{i=0}^{(p-1)/2}\omega^{y2^m+i\,y2^{r+m}} + \sum_{j=0}^{(p-3)/2} \omega^{y2^m - (\frac{p-1}{2}-j)y2^{r+m}}\\
 &=&\omega^{y2^m} \left( \sum_{i=0}^{(p-1)/2}\omega^{(2i)\,y2^{r+m-1}} + \sum_{j=0}^{(p-3)/2} \omega^{(2j+1)y2^{r+m-1}} \right) \\
 &=&\omega^{y2^m}\sum_{i=0}^{p-1} \omega^{i\,y 2^{r+m-1}}\\
 &=&\omega^{y2^m} \sum_{i=0}^{p-1} \omega^{i 2^{r}} \quad \quad \mbox{(by Lemma \ref{lem:sumhelp})}\\
&=& 0.
\end{eqnarray*}
Therefore, $\lambda_{x, \alpha} = \mu_{x+\frac{n}{2}, \beta}$.  Similarly, $\mu_{x+\frac{n}{2},\alpha} = \lambda_{x,\beta}$.  Thus, $\lambda_x = \mu_{x+\frac{n}{2}}$.
\end{proof}

\subsection{No Repeated Eigenvalues}

The next goal is to prove that these graphs have no repeated eigenvalues in their spectra.  This is needed in section \ref{sec:noniso} to show that the graphs are not isomorphic.  In this section we will again be relying heavily on on the group ring, $\mathbb{Z}G$, and homomorphism, $\varphi$, from section \ref{sec:groupring}.  Since we have proved in the previous section that the graphs have the same spectrum, we only need to prove that one of the graphs has no repeated eigenvalues.

\begin{thm}\label{thm:norepeig}
Let $n = 2^rp$ where $r$ is an integer such that $r \geq 2$ and $p$ is any odd prime, and let 
\begin{equation*}
A = \{ 1+i2^r \mid 0 \leq i \leq (p-1)/2 \} \cup \{1+j2^r + p2^{r-1} \mid 1 \leq j \leq (p-1)/2 \}.
\end{equation*}
If $X = \cay{\zn, A}$, then $X$ has no repeated eigenvalues.
\end{thm}

\begin{proof}
Order the eigenvalues of $X$ so that the $x^{th}$ eigenvalue of the spectrum of $X$ is 
\begin{equation}
\lambda_x = \sum_{i=0}^{(p-1)/2} \omega^{x(1+i2^r)} + \sum_{j=1}^{(p-1)/2} \omega^{x(1+j2^r+n/2)}.
\end{equation}

Suppose that there is some $y$ such that $\lambda_x = \lambda_y$ (in order to show that $x \equiv y \mod n$).  Therefore, $\lambda_x - \lambda_y =\lambda_x + \omega^{n/2}\lambda_y = 0$.  Let $\alpha \in \mathbb{N}G$ be defined by
\begin{equation} \label{alpha}
\alpha =  \sum_{i=0}^{(p-1)/2} \left( z^{x(1+i2^r)} + z^{y(1+i2^r)+n/2} \right) +  \sum_{j=1}^{(p-1)/2} \left( z^{x(1+j2^r+n/2)}+ z^{y(1+j2^r+n/2) + n/2}\right) .
\end{equation}
For the rest of this proof, let $\alpha = \sum_{k=0}^{n-1}C_k z^k$ be the normal form of $\alpha$. 

 Since $\varphi (\alpha) = \lambda_x + \omega^{n/2} \lambda_y $, we know that $\alpha \in \mathbb{N}G \cap \ker(\varphi)$.  By Lemma \ref{lem:corlam}, $\alpha$ must also be an element of $\mathbb{N}G\, \sigma (H_2) + \mathbb{N}G \, \sigma (H_p)$ where $H_2$ and $H_p$ are the unique subgroups of $G$ of size 2 and $p$, respectively.  Thus, we can write
\begin{equation}\label{eq:agbg}
\alpha = \sum_{g \in G} a_g\, g\, \sigma(H_2) + \sum_{g \in G}b_g \, g\, \sigma(H_p),
\end{equation}
where $a_g$, $b_g \in \mathbb{N}$.  Therefore, 
\begin{equation*}
\varepsilon(\alpha) = \sum_{g \in G}( 2a_g + p b_g) .
\end{equation*}
However, we defined $\alpha$ by an explicit formula (see equation \ref{alpha}) and can calculate the exact value of $\varepsilon(\alpha)$.  Namely $\varepsilon(\alpha) = \frac{p+1}2 \cdot 2 + \frac{p-1}2 \cdot 2 = 2p$.  Therefore, 
\begin{equation} \label{eq:inone}
\sum_{g \in G}( 2a_g + p b_g) = 2p.
\end{equation}
 So, either $a_g =0$ for all $g \in G$ or $b_g = 0$ for all $g \in G$.  This implies that either $\alpha \in \mathbb{N}G\, \sigma(H_2)$ or $\alpha \in \mathbb{N}G\, \sigma(H_p)$.

We will consider cases based on whether $p$ and 2 divide $x$.  In each case we will see that $\alpha$ must be an element of $\mathbb{N}G\, \sigma(H_2)$ and then that $x\equiv y \mod n$.\\


\emph{Case 1.}  $x$ is odd. \quad In this case, $z^{x(1+i2^r) +n/2} = z^{x(1+i2^r +n/2)}$ for all $i$.  Therefore, 
\begin{equation*}
\sum_{i=1}^{(p-1)/2} z^{x(1+i2^r)}+ z^{x(1+i2^r +n/2)} = \sum_{i=1}^{(p-1)/2} z^{x(1+i2^r)}\sigma(H_2)
\end{equation*}
Using the notation of Equation \ref{eq:agbg}, we can see that $a_{ z^{x(1+i2^r)}}$ is at least one.  Thus, $b_g$ must be zero for all $g$ and we can conclude that $\alpha \in \mathbb{N}G \, \sigma (H_2)$.
Let $\beta$ be defined by
\begin{eqnarray}
\nonumber \beta& =& \alpha - \sum_{i=1}^{(p-1)/2} z^{x(1+i2^r)}+ z^{x(1+i2^r +n/2)} \\
\label{eq:beta} &=& z^x + \sum_{i=0}^{(p-1)/2} z^{y(1+i2^r)+n/2} +  \sum_{j=1}^{(p-1)/2} z^{y(1+j2^r+n/2) + n/2}.
\end{eqnarray}
Since $\beta$ is the difference of two elements of $\mathbb{N}G \, \sigma (H_2)$, we know that $\beta$ must be an element of $\mathbb{Z}G \, \sigma (H_2)$.  Let $\beta = \sum_{k=0}^{n-1}B_k z^k$ be the normal form of $\beta$.  We can see that $B_x \geq 1$.  By Lemma \ref{lem:const}, we know that $B_{x+n/2} \geq 1$ as well.  Therefore, $z^{x+n/2} = z^{y(1+i2^r) +n/2}$ for some $0 \leq i \leq (p-1)/2$ or $z^{x+n/2} = z^{y(1+j2^r +n/2)+n/2}$ for some $1 \leq j \leq (p-1)/2$.  Which is to say,
\begin{equation*}
x \equiv y(1+i2^r) \mbox{  or  }y(1+j2^r+n/2) \mod n.
\end{equation*}
Therefore, $y$ must be odd as well, and we can conclude that $z^{y(1+i2^r)} = z^{y(1+i2^r+n/2) +n/2}$ for all $i$.  Then, 
\begin{equation*}
\sum_{i=1}^{(p-1)/2} z^{y(1+i2^r +n/2) +n/2} + z^{y(1+i2^r) +n/2} = \sum_{i=1}^{(p-1)/2} z^{y(1+i2^r)} \sigma(H_2) \in \mathbb{N}G \, \sigma(H_2).
\end{equation*}
Thus, \begin{equation*} \beta - \sum_{i=1}^{(p-1)/2} z^{y(1+i2^r +n/2) +n/2} + z^{y(1+i2^r) +n/2} =z^x +z^{y+n/2} \in \mathbb{Z}G \, \sigma(H_2).
\end{equation*}
By Lemma \ref{lem:const}, we can conclude that $z^{x+n/2} = z^{y+n/2}$, and therefore, $x \equiv y \mod n$.

\emph{Case 2}. $2|x$ \emph{and} $p|x$. \quad In this case we have 
\begin{eqnarray}
\nonumber \alpha  &=&  \sum_{i=0}^{(p-1)/2}z^{x} + z^{y(1+i2^r)+n/2} +  \sum_{j=1}^{(p-1)/2}z^{x}+ z^{y(1+j2^r+n/2) + n/2}\\
&=& p\, z^x + \sum_{i=0}^{(p-1)/2} z^{y(1+i2^r)+n/2} +  \sum_{j=1}^{(p-1)/2} z^{y(1+j2^r+n/2) + n/2}. \label{eq:p2}
\end{eqnarray}
We can see that $C_x \geq p$.  
Since $i \not\equiv j \mod p$ implies that $x +i 2^r \not\equiv x + j2^r \mod n$, we know that $C_{x+i2^r}$ refers to a distinct coefficient for all $0 \leq i < p$.  Therefore, we can conclude that
\begin{equation*}
\varepsilon (\alpha ) \geq \sum_{i=0}^{p-1}C_{x+i2r}.
\end{equation*}

If $\alpha \in \mathbb{N}G \, \sigma(H_p)$, then we could conclude that
\begin{eqnarray*}
\varepsilon (\alpha ) &\geq& \sum_{i=0}^{p-1}C_{x+i2r} \\
&=& \sum_{i=0}^{p-1}C_{x} \quad \mbox{ (by Lemma \ref{lem:const})}\\
&\geq & \sum_{i=0}^{p-1}p \\
&=& p^2\\
&>& 2p.
\end{eqnarray*}
This is a contradiction because we already know that $\varepsilon (\alpha)  = 2p$.  Therefore, $\alpha  \notin \mathbb{N}G \, \sigma (H_p)$, and we can assume that $\alpha \in \mathbb{N}G \, \sigma(H_2)$.

Since $C_x \geq p$, Lemma \ref{lem:const} tells us that $C_{x+n/2} \geq p$.  From Equation \ref{eq:p2}, we can see that for this to be true, $z^{x +n/2} = z^{y(1+i2^r)+n/2}$ for all $0 \leq i \leq (p-1)/2$ and $z^{x+n/2} = z^{y(1+j2^r+n/2) + n/2}$ for all $1 \leq j \leq (p-1)/2$.  Thus, $y(1+j2^r+n/2)+n/2 \equiv y(1+i2^r)+n/2 \mod n$, and $yp2^{r-1} \equiv y2^r(i-j) \mod n$ for all $0 \leq i \leq (p-1)/2$ and $1 \leq j \leq (p-1)/2$.  Therefore, $p$ and 2 must divide $y$ and we can then rewrite $\alpha$ as
\begin{eqnarray*}
\alpha &=& p\, z^x + \sum_{i=0}^{(p-1)/2} z^{y(1+i2^r)+n/2} +  \sum_{j=1}^{(p-1)/2}+ z^{y(1+j2^r+n/2) + n/2}\\
&=& p(z^x + z^{y+n/2}).
\end{eqnarray*}
Since $\alpha \in  \mathbb{N}G \, \sigma(H_2)$, $z^{x+n/2} = z^{y+n/2}$.  Hence, $x \equiv y \mod n$.\\
\\

\emph{Case 3}. $2|x$ \emph{and} $p \nmid x$. \quad In \emph{Case 1} we saw that if $x$ is odd, then $y$ must also be odd.  Since $x$ and $y$ were chosen arbitrarily, we can assume that $y$ must be even in this case.  In \emph{Case 2}, we saw that if $x$ is even and $p$ divides $x$, then $p$ must divide $y$.  Again, since $x$ and $y$ were chosen arbitrarily, we can assume that $p$ does not divide $y$ in this case.  Therefore, we have 
\begin{equation}\label{eq:newalpha}
\alpha  = z^x + z^{y+n/2} + \sum_{i=1}^{(p-1)/2}2z^{x(1+i2^r)} + 2z^{y(1+i2^r)+n/2}.
\end{equation}

Suppose that $\alpha \in \mathbb{N}G\, \sigma( H_p)$ (in order to arrive at a contradiction).  Since $C_{x(1+2^r)} \geq 2$, Lemma \ref{lem:const} tells us that $C_{x(1+2^r)+i2^r}\geq 2$ for all $i \in \mathbb{Z}_p$.  Since $(x,p) = 1$ we know that for all $0\leq j <p$ there exists $0 \leq i <p$ such that $i\equiv jx-x \mod n$.  Therefore, $x(1+2^r)+i2^r \equiv x(1+j2^r) \mod n$ for some $i \in \mathbb{Z}_p$.  Therefore, we can say that $C_{x(1+j2^r)} \geq 2$ for all $0 \leq j<p$.
If
\begin{equation*}
x(1+i2^r) \equiv x(1+j2^r) \mod n 
\end{equation*}
for $i \neq j \mod p$, then
\begin{equation*}
\mbox{then } 0  \equiv x2^r(j-i) \mod n.
\end{equation*}
This is a contradiction because $p$ does not divide $x$.  Therefore, the coefficients $C_{x(1+i2^r)}$ are referring to unique terms for each $0 \leq i <p$.  Then, we know
\begin{eqnarray*}
\varepsilon(\alpha) &\geq& \sum_{i=0}^{p-1} C_{x(1+i2^r)}\\
&=& \sum_{i=0}^{p-1} C_{x} \quad \mbox{ (by Lemma \ref{lem:const})}\\
&\geq& \sum_{i=0}^{p-1} 2\\
&=& 2p.
\end{eqnarray*}
Since we know that $\varepsilon (\alpha) = 2p$, we know that all inequalities must be equalities.  This implies that 
\begin{eqnarray*} 
\sum_{i=0}^{p-1} C_{x(1+i2^r)} & = &\varepsilon(\alpha)\\
&=& \varepsilon \left( \sum_{i=0}^{p-1} 2z^{x(1+i2^r)} \right) + \varepsilon \left( \alpha - \sum_{i=0}^{p-1} 2z^{x(1+i2^r)} \right)\\
&=& \sum_{i=0}^{p-1} C_{x(1+i2^r)} +  \varepsilon \left( \alpha - \sum_{i=0}^{p-1} 2z^{x(1+i2^r)} \right)\\
\end{eqnarray*}
Therefore, $ \varepsilon \left( \alpha - \sum_{i=0}^{p-1} 2z^{x(1+i2^r)} \right) =0$ and
\begin{equation}
\label{eq:alphnot}
\alpha = \sum_{i=0}^{p-1}2z^{x(1+i2^r)}.
\end{equation}  

Since $(p, 2^r) =1$ there exist $k$ and $\ell$ such that $kp = 1+\ell 2^r$.  We will choose $k$ and $\ell$ such that $0 < \ell < p$.  We will now examine sub-cases based on the size of $\ell$.
\\

\emph{Sub-case 3.1}.  $\ell \leq (p-1)/2$. \quad By Equation \ref{eq:newalpha}, we know that $C_{y(1+\ell 2^r)+n/2}\geq 2$.  So, by Equation \ref{eq:alphnot}, we know that $z^{y(1+\ell 2^r)+n/2} = z^{x(1+i2^r)}$ for some $0 \leq i < p$.  If $z^{y(1+\ell 2^r)+n/2} = z^{x(1+ \ell 2^r)}$, then Equation \ref{eq:newalpha} shows us that $C_{x(1+ \ell 2^r)} \geq 3$.  This is a contradiction to Equation \ref{eq:alphnot} since we have already established that $x(1+i2r) \not\equiv x(1+j2^r) \mod n$ whenever $i \not\equiv j \mod p$ .  Therefore $z^{y(1+\ell 2^r)+n/2} = z^{x(1+i2^r)}$ for some $i \neq \ell \mod p$.  This is to say that
\begin{equation*}
y(1+ \ell 2^r) + n/2 \equiv p(yk + 2^{r-1})\equiv x(1+ i2^r) \mod n.
\end{equation*}
We know that $p$ cannot divide $x$, and if p divides $1+i2^r$, then $i$ must be congruent to $\ell$.  Therefore, we have arrived a contradiction and we can conclude that when $0 < \ell \leq (p-1)/2$, $\alpha \notin \mathbb{N}G \, \sigma(H_p)$.
\\

\emph{Sub-case 3.2}. $(p-1)/2 < \ell < p$. \quad By Equation \ref{eq:alphnot}, we know that $C_{x(1+\ell 2^r)} = 2$.  Therefore, by Equation \ref{eq:newalpha}, we know that $z^{x(1+\ell 2^r)}$ is equal to $z^{x(1+i2^r)}$ or $z^{y(1+i2^r)+n/2}$ for some $0 \leq i \leq(p-1)/2$.  In either case, this would imply that $p$ divides $(1+i2^r)$.  This a contradiction since $i$ cannot be congruent to $\ell$ mod $p$.  Therefore, in both sub-cases, $\alpha \notin \mathbb{N}G \, \sigma(H_p)$.
\\

We can now assume that $\alpha \in \mathbb{N}G \, \sigma(H_2)$.  Since $x(1+i2^r)+a\frac{n}{2} \not\equiv x(1+j2^r) +b \frac{n}{2}\mod n$ whenever $i \not\equiv j \mod p$ for any $a,b \in \{0,1\}$ we can assume that the coefficients $C_{x(1+i2^r)}$ and $C_{x(1+i2^r) +n/2}$ are referring to unique terms for all $0 \leq i \leq (p-1)/2$.  Thus, for some $\beta \in \mathbb{N}G \, \sigma(H_2)$, we can write
\begin{eqnarray}\label{eq:sthere}
\alpha & = & \sum_{i=0}^{(p-1)/2} \left( C_{x(1+i2^r)}z^{x(1+i2^r)} +C_{x(1+i2^r) +n/2}z^{x(1+i2^r) +n/2} \right) + \beta\\
\nonumber &=& \sum_{i=0}^{(p-1)/2} C_{x(1+i2^r)} \left( z^{x(1+i2^r)} +z^{x(1+i2^r) +n/2} \right) + \beta \quad \mbox{ (by Lemma \ref{lem:const})}.
\end{eqnarray}
Therefore, 
\begin{eqnarray*}
2p = \varepsilon (\alpha) &=& \sum_{i=0}^{(p-1)/2}\left( 2 \,C_{x(1+i2^r)} \right) + \varepsilon (\beta)\\
 &\geq& 2 + 2 \cdot 2 \frac{(p-1)}{2} + \varepsilon(\beta)\quad \mbox{ (from Equation \ref{eq:newalpha}) }\\
&=& 2p + \varepsilon(\beta).
\end{eqnarray*}
Which implies that $\varepsilon(\beta) = 0$ and all inequalities must be equalities.  We can conclude that 
\begin{equation*}
\alpha =\sum_{i=0}^{(p-1)/2} C_{x(1+i2^r)} \left( z^{x(1+i2^r)} +z^{x(1+i2^r) +n/2} \right)\end{equation*}
and
\begin{equation*}
\sum_{i=0}^{(p-1)/2} C_{x(1+i2^r)} = 1 + 2 \frac{(p-1)}{2}.
\end{equation*}
Looking again at Equation \ref{eq:newalpha}, we can conclude that for these equalities to be true, $C_x = 1$ and $C_{x(1+i2^r)} = 2$ for all $1\leq i \leq{(p-1)/2}$.  Thus, $C_k =1$ iff $z^k = z^x$ or $z^k = z^{x+n/2}$.

We can now repeat the same process focusing on the $y$-terms instead of the $x$-terms.  Since $C_{y(1+i2^r)}$ and $C_{y(1+i2^r)+n/2}$ are referring to distinct terms for all $0 \leq i \leq (p-1)/2$, we can write
\begin{eqnarray*}
\alpha & = & \sum_{i=0}^{(p-1)/2} \left( C_{y(1+i2^r)}z^{y(1+i2^r)} +C_{y(1+i2^r) +n/2}z^{y(1+i2^r) +n/2} \right) + \gamma\\
&=& \sum_{i=0}^{(p-1)/2} C_{y(1+i2^r)+n/2} \left( z^{y(1+i2^r)} +z^{y(1+i2^r) +n/2} \right) + \gamma
\end{eqnarray*}
for some $\gamma$.  Using the same logic from Equation \ref{eq:sthere} onward, we will conclude that $C_k=1$ iff $z^k = z^{y+n/2}$ or $z^k = z^y$.  Therefore, $z^x$ is equal to $z^y$ or $z^{y+n/2}$.  If $z^x =z^{y+n/2}$, then $C_x \geq 2$.  This is a contradiction.  It must be the case that, $z^x = z^y$. Therefore, $x \equiv y \mod n$ in all three cases.
\end{proof}

\subsection{Non-Isomorphic}\label{sec:noniso}

In 1967, \'{A}d\'{a}m made the conjecture that $\cay{\zn, S_1}$ and $\cay{\zn, S_2}$ are isomorphic iff $S_1 = q S_2$ where $(q,n)=1$ and $qS_2 = \{ qs \mid s \in S_2 \}$ \protect \cite{Adam}.  In 1969, Elspas and Turner showed that \'Ad\'am's conjecture was true if $\cay{\zn, S_1}$ and $\cay{\zn, S_2}$ have no repeated eigenvalues \protect \cite{Elspas}.  Since we have just seen that the graphs defined in this chapter have no repeated eigenvalues, \'{A}d\'{a}m's conjecture holds.  Thus, all we need to show is that our graphs' connection sets are not equivalent by multiplication by a number relatively prime to $n$.

\begin{lem}\label{rmrk:nonequiv}
Let $n=2^rp$, where $p$ is an odd prime and $2\leq r $.  Let $A$ and $B$ be sets that depend on $r$ and $p$ as follows:
\begin{eqnarray*}
A&=&\{1+i2^r \mid  0\leq i \leq \frac{p-1}2 \} \cup \{1+j2^r +\frac{n}{2} \mid  1 \leq j \leq \frac{p-1}{2} \} \\
B &=& \{1-i2^r \mid  0\leq i \leq \frac{p-1}{2} \} \cup \{1-j2^r +\frac{n}{2} \mid  1 \leq j \leq \frac{p-1}{2} \}
\end{eqnarray*}
One of these sets will be comprised of numbers that are all relatively prime to $n$ and the other set will contain exactly two values that are divisible by $p$.
\end{lem}

\begin{proof}
For this proof, it is helpful to rewrite $B$ as the equivalent set mod $n$:
\begin{equation*}
B = \{1+i2^r \mid  \frac{p+1}2 \leq i \leq p \} \cup \{1+j2^r +\frac{n}{2} \mid  \frac{p+1}2 \leq j \leq p-1 \}.
\end{equation*}
Since $(p, 2^r) =1$ there exist $k$ and $\ell$ such that $kp = 1+\ell 2^r$.  We will choose $k$ and $\ell$ such that $0 < \ell < p$.  The number $1+\ell 2^r$ will be an element of either $A$ or $B$ depending on whether or not $\ell$ is greater than $(p-1)/2$.   We can also conclude that $1+\ell 2^r +n/2$, which will be in the same set as $1 + \ell 2^r$, is also divisible by $p$.  Thus, we can see that one of the sets will have at least two elements that are divisible by $p$.  Furthermore,
\begin{equation*}
1+ i 2^r  \equiv 1+ i 2^r + \frac{n}2 \not\equiv 0 \mod p\\
\end{equation*}
whenever $i \not\equiv \ell \mod p$.  Therefore, there can be no other elements of either set that are divisible by $p$.  Since all of the elements in both $A$ and $B$ are odd, we can conclude that all of the elements in both $A$ and $B$ besides $1+\ell 2^r$ and $1+ \ell 2^r +n/2$ are relatively prime to $n$.
\end{proof}
By this lemma, we can see that $A$ and $B$ cannot be equivalent via multiplication by a number relatively prime to $n$, and therefore, the results of Elspas and Turner mentioned above tell us that the circulant graphs of order n with connection sets $A$ and $B$ must not be isomorphic. We have now proved that the construction  described in the beginning of this section creates isospectral, non-isomorphic graphs.

\subsection{Extending the Construction}\label{sec:extension}

We can use the connection sets described earlier to create even more isospectral, circulant graphs of order $n$ where $n = 2^rp$ for some prime $p$.  Letting $A$ and $B$ be as define in Section \ref{sec:def}, we can create the new connection sets as follows:
\begin{eqnarray*}
\tilde{A} &=& A \cup qA\\
\tilde{B} &=& B \cup qB
\end{eqnarray*}
where $q$ is relatively prime to $n$ and $qA = \{ q\, a \mid a \in A \}$.  Now, we can use these connection sets to create two new graphs (or multigraphs), $\tilde{X} = \cay{\zn, \tilde{A}}$ and $\tilde{Y} = \cay{\zn, \tilde{B}}$. 

\begin{lem}\label{lem:newspectrum}
The graphs described above, \begin{eqnarray*} \tilde{X}=\cay{\zn, A \cup qA} \\ \tilde{Y} = \cay{\zn, B \cup qB}, \end{eqnarray*} have the same spectrum.
\end{lem}

\begin{proof}
Order the eigenvalues of $\tilde{X}$ as follows: let the $x^{th}$ eigenvalue of $\tilde{X}$ be 
\begin{equation*}
\tilde{\lambda}_x = \sum_{\tilde{a}\in \tilde{A}} \omega^{x \, \tilde{a}}
\end{equation*}
where $\omega$ is a primitive $n^{th}$ root of unity.  Letting $X=\cay{\zn, A}$, as described in Section \ref{sec:def}, and letting $\lambda_x$ be the $x^{th}$ eigenvalue of $X$ by the ordering described in Section \ref{sec:isospec}, we can see that 
\begin{eqnarray*}
\tilde{\lambda}_x &=& \sum_{\tilde{a}\in \tilde{A}} \omega^{x \, \tilde{a}}\\
& = &\sum_{a \in A} \omega^{x \, a} + \omega^{x \, qa}\\
&=& \lambda_x + \lambda_{qx}.
\end{eqnarray*}
Similarly, we can order the spectrum of $\tilde{Y}$ such that the $x^{th}$ eigenvalue is
\begin{equation*}
\tilde{\mu}_x = \mu_x + \mu_{qx}
\end{equation*}
where $\mu_x$ is the $x^{th}$ eigenvalue of the spectrum of $Y$ under the ordering described in Section \ref{sec:isospec}.  Thus, we have written the eigenvalues of $\tilde{X}$ and $\tilde{Y}$ in terms of the eigenvalues mentioned in Lemma \ref{rmrk:samespec}, and we can use the results of that lemma.  
\end{proof}

If we let $q=-1$ then we have two undirected graphs.  When $r>2$, the undirected graphs do not have any double edges (they are not multigraphs).  Thus, we can create a pair of undirected, isospectral circulant graphs.  These graphs can have repeated eigenvalues, and therefore we cannot use the same process used in section \ref{sec:noniso} to prove that they are not isomorphic.  However, we can prove that some of them are not isomorphic.  Musychuk proved that \'Ad\'am's conjecture (as described in the previous section) holds for graphs on $n$ vertices when either $n$, $n/2$ or $n/4$ is an odd, square-free number \protect \cite{Muz17}, \protect \cite{Muz18}.   It will still be the case that one of the connections sets (either $\tilde{A}$ or $\tilde{B}$) will contain values divisible by $p$, and the other connection set will be comprised entirely of values that are relatively prime to $n$.  Therefore, when $n=2^2p$ for any odd prime $p$, the graphs (which are actually multigraphs in this case) cannot be isomorphic.  As far as the rest of the graphs are concerned (namely, when $n=2^rp$ for $r>2$), it would be just as interesting to prove that these graphs are isomorphic as it would be to prove that they are not.  Thus, we are left with the following open problem:
\begin{qstn}
Are the graphs $\tilde{X}$ and $\tilde{Y}$ described in this section isomorphic for any values of $r$?
\end{qstn}


\begin{thebibliography}{00}

\bibitem{Adam}
{\'A}d{\'a}m, A.,
\emph{Reseach problem},
J. Combinatorial Theory 2,
1967,
229-230
\\

\bibitem{Bab}
Babai, L.,
\emph{Spectra of Cayley graphs}
J. Combin. Theory Ser. B \textbf{27},
1979, no. 2, 180-189
\\

\bibitem{Spec}
Cvetkovi{\'c}, D., Rowlinson, P., \& Simi{\'c}, S.,
\emph{Eigenspaces of graphs},
encyclopedia of Mathematics and its Applications, vol. 66,
Cambridge University Press, 
Cambridge, 1997
\\

\bibitem{Cv2}
Cvetkovi{\'c}, D. M.,
\emph{Graphs and their spectra},
Univ. Beograd. Publ. Elektrotehn. Fak. Ser. Mat. Fiz.,
1971, no. 354 -356
\\

\bibitem{Chem}
D'Amato, S. S., Gimarc, B. M., \& Trinajsti{\'c}, N.,
\emph{Isospectral and subspectral molecules},
Croat. Chem. Acta. \textbf{54},
1981,
no. 1, 1-52
\\

\bibitem{Ram}
Davidoff, G., Sarnak, P.,  \& Valette, A.,
\emph{Elementary number theory, group theory, and Ramanujan graphs},
London Mathematical Society Student Texts, 
vol. 55,
Cambridge University Press,
Cambridge, 2003
\\

\bibitem{Elspas}
Elspas, B., \& Turner, J.,
\emph{Graphs with circulant adjacency matrices}
J Combinatorial Theory 9,
1970,
297-307
\\

\bibitem{Gods-small}
Godsil, C., Holton, D., \& McKay, B.,
\emph{The spectrum of a graph},
Combinatorial mathematics, V,
Spriner, Berlin,
1977, pp. 91-117. 
Lecture Notes in Math., Vol. 622
\\

\bibitem{Gods-const}
Godisl, C. \& McKay, B. D.,
\emph{Constructing cospectral graphs},
Aequationes Math. \textbf{25},
1982,
no. 2-3, 257-268
\\

\bibitem{Gods1}
Godsil, C. \& Royle, G.,
\emph{Algebraic graph theory},
Graduate Texts in Mathematics,
vol. 207, Springer-Verlag, New York, 2001.
\\

\bibitem{Iso27}
Cai Heng Li,
\emph{Finite CI-groups are soluble},
Bull. London Math. Soc. \textbf{31},
1999, no. 4, 419-423
\\

\bibitem{Lam}
Lam, T. Y. \& Leung, K. H.,
\emph{On vanishing sums of roots of unity},
J. Algebra \textbf{224},
2000, no. 1,
91-109
\\

\bibitem{Lub}
Lubotzky, A., Samuels, B., \& Bishne, U.,
\emph{Isospectral Cayley graphs of some finite simple groups},
duke Math. J. \textbf{135}, 2006,
no. 2, 381-393
\\

\bibitem{Iso18}
Muzychuk, M., Klin, M., \& P{\"o}chel, R.,
\emph{The isomorphism problem for circulant graphs via Shur ring theory},
Codes and association schemes (Piscataway, NJ, 1999),
DIMACS Ser. Discrete Math. Theoret. Comput. Sci., vol. 56,
Amer. math. Soc., Providence, RI, 2001,
pp. 241-246.
\\

\bibitem{Muz17}
Muzychuk, M.,
\emph{\'{A}d\'am's conjecture is true in the square-free case},
J. combin. Theory Ser. A \textbf{72}, 1995, no. 1, 118-134
\\

\bibitem{Muz18}
Muzychuk, M.,
\emph{On \'{A}d\'am's conjecture for circulant graphs},
Discrete Math. \textbf{176},
1997, no. 1-3,
 \\
 
 \bibitem{Iso}
 Muzychuk, M.,
 \emph{A solution of the isomorphism problem for circulant graphs},
 Proc. London Math. Soc. (3) \textbf{88}, 2004, no. 1, 1-41
 \\
 
 \bibitem{Iso35}
 P\'{a}lfy, P. P.,
 \emph{Isomorphism problem for relational structures with a cyclic automrophism},
 European J. Combin. \textbf{8}, 1987, no. 1, 1-41
 
 
 
\end{thebibliography}
\end{document}